\pdfoutput=1
\documentclass{amsart}

\usepackage{graphicx}

\usepackage{amscd}

\newtheorem{theorem}{Theorem}[section]

\newtheorem{proposition}[theorem]{Proposition}
\newtheorem{assumption}[theorem]{Assumption}
\theoremstyle{definition}

\theoremstyle{remark}
\newtheorem{remark}[theorem]{Remark}

\numberwithin{equation}{section}
\newcommand{\genindex}{\bullet}
\newcommand{\hodge}{\ast}
\newcommand{\xqedhere}[2]{%
  \rlap{\hbox to#1{\hfil\llap{\ensuremath{#2}}}}}
\begin{document}

\title[Structure-preserving mesh coupling]{Structure-preserving mesh coupling based on the Buffa-Christiansen complex}


%    author one information
\author[O.~Niemim\"aki]{Ossi Niemim\"aki}
\address[O.~Niemim\"aki]{Tampere University of Technology\\
		DEE -- Electromagnetics\\
		P.O. Box 692\\
		33101 Tampere\\
		Finland}
\curraddr{University of Helsinki\\
		 Department of Mathematics and Statistics\\
		 P.O. Box 68\\
		 00014 University of Helsinki\\
		 Finland}
\email{ossi.niemimaki@helsinki.fi}
%\thanks{}

\author[S.~Kurz]{Stefan Kurz}
\address[S.~Kurz]{Tampere University of Technology\\
		DEE -- Electromagnetics\\
		P.O. Box 692\\
		33101 Tampere, Finland}
\curraddr{Technische Universit\"at Darmstadt\\
		 Graduate School Computational Engineering\\
		 Dolivostra{\ss}e 15\\
		 64293 Darmstadt\\
		 Germany}
\email{kurz@gsc.tu-darmstadt.de}
%\thanks{}

\author[L.~Kettunen]{Lauri Kettunen}
\address[L.~Kettunen]{Tampere University of Technology\\
		DEE -- Electromagnetics\\
		P.O. Box 692\\
		33101 Tampere\\
		Finland}
%\curraddr{}
\email{lauri.kettunen@tut.fi}
%\thanks{}

%    author two information
%\author{}
%\address{}
%\curraddr{}
%\email{}
%\thanks{}

%    \subjclass is required.
%\subjclass[2010]{Primary 65N30, 78M10.}

\date{\today}

%\dedicatory{}

%    Abstract is required.
\begin{abstract}
The state of the art for mesh coupling at nonconforming interfaces is presented and reviewed. Mesh coupling is frequently applied to the modeling and simulation of motion in electromagnetic actuators and machines. The paper exploits Whitney elements to present the main ideas. Both interpolation- and projection-based methods are considered. In addition to accuracy and efficiency, we emphasize the question whether the schemes preserve the structure of the de Rham complex, which underlies Maxwell's equations. As a new contribution, a structure-preserving projection method is presented, in which Lagrange multiplier spaces are chosen from the Buffa-Christiansen complex. Its performance is compared with a straightforward interpolation based on Whitney and de Rham maps, and with Galerkin projection.\\
\end{abstract}

{\let\thefootnote\relax\footnotetext{Some figures are omitted due to a restricted copyright. Full paper to appear in Mathematics of Computation.}}

\maketitle
\section{Introduction}
This paper deals with mesh coupling at nonconforming interfaces. Such situations frequently occur in the modelling of motion, when different submeshes are sliding with respect to each other. More generally, mesh coupling can relax strict conformity requirements during mesh generation, or combine submodels that have been created independently. Our motivation comes from computational electromagnetics, where the most common coupling methods are based either on interpolation \cite{Niu2012,Shi2008}, or on projection by means of Lagrange multipliers \cite{Lange2012,Rodger1990}, also known as mortar element methods \cite{Rapetti2000}.\par
Mesh coupling methods should be accurate, efficient and preserve the key structures. Accurate in that their convergence rate should not deteriorate the convergence of the numerical schemes in the adjacent domains, and efficient so that the additional numerical effort should be acceptable. Specifically in the context of electromagnetics, the structure of the de Rham complex should be preserved, for it is crucial to Maxwell's equations. For example, the property of being a gradient field should be retained while passing from one mesh to another. This is comparable to geometric multigrid, where structure-preserving restriction and prolongation operators were discussed, e.g., in \cite{Clemens2004a}.
\section{Setting}
\begin{figure}
\centering
\includegraphics[width=0.4\linewidth]{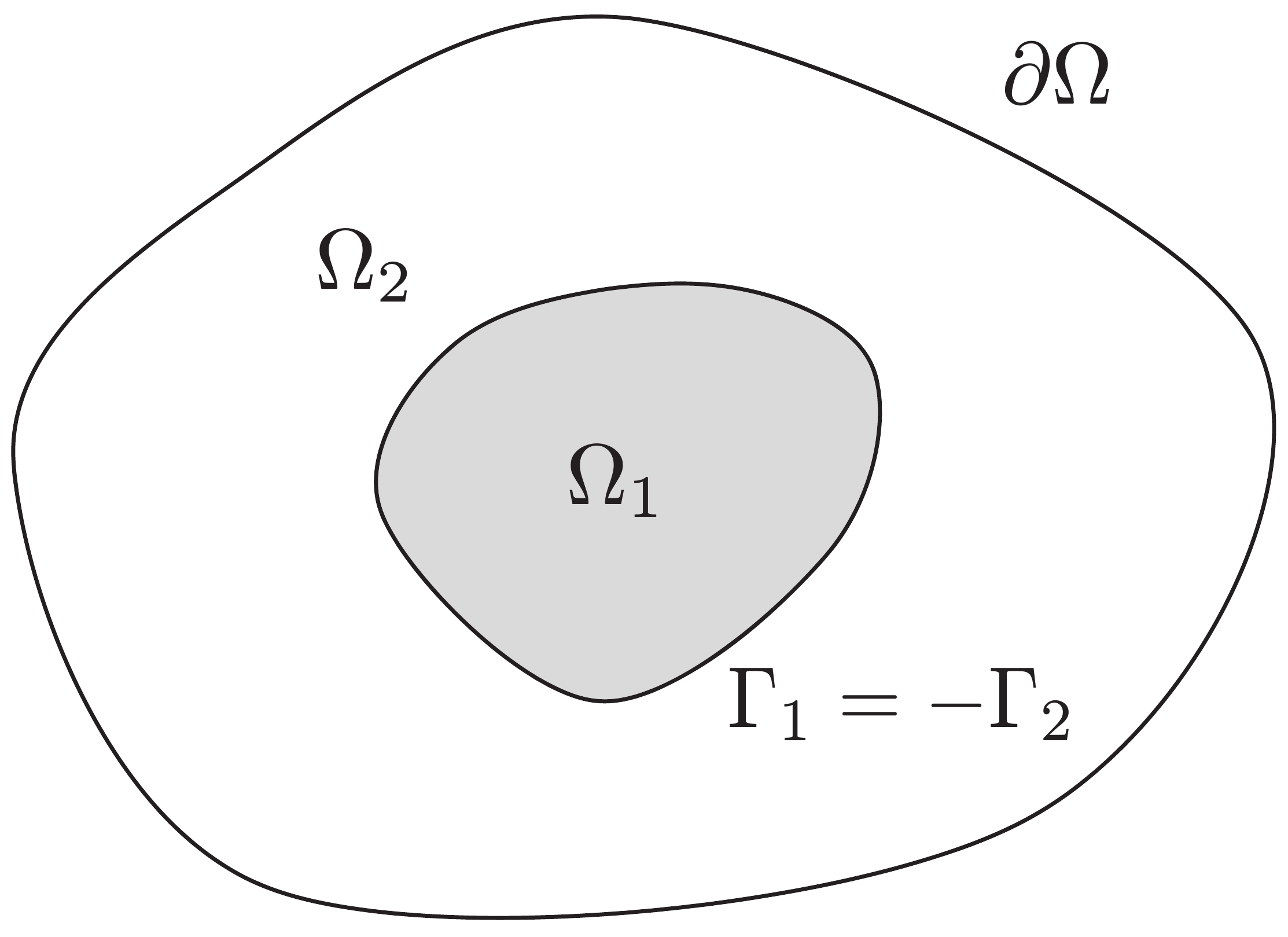}
\caption{\label{fig_Model_Onion}The model problem consists of 3D domain $\Omega$, which is partitioned in two nonoverlapping sub-domains $\Omega_1$ and $\Omega_2$. For simplicity, we assume an onion-type partitioning, i.e., $\partial\Omega$ and $\partial\Omega_1$ are disjoint. For the coupling interface, we let $\Gamma_1=\partial\Omega_1$, $\Gamma_2=\partial\Omega_2\setminus\partial\Omega$. The partitioning is geometrically conforming, such that in terms of boundary orientation $\Gamma_1=-\Gamma_2$. The interface is a curvilinear Lipschitz polyhedron.}
\end{figure}
The setting of the three-dimensional model problem is depicted in Fig.~\ref{fig_Model_Onion}. We consider the trace spaces $\mathcal{W}^r_i$ of Whitney $r$-forms ($r=0,1,2$) on the interface related to domain $\Omega_i$ ($i=1,2$). The cases $r=0,1,2$ correspond to nodal, edge, and facet elements, respectively. The discretizations are assumed nonconforming, that is, $\mathcal{W}^r_1\ne\mathcal{W}^r_2$. We are looking for a family of {\em mesh coupling operators} $Q^r_{ji}:\mathcal{W}^r_i\to\mathcal{W}^r_j$, $ji\in\{12;21\}$, such that $Q^r_{ij}\circ Q^r_{ji}$ converges to the identity mapping in some well-defined sense as the meshes are refined ({\em approximation property}). Moreover, we require $\mathrm{d}\circ Q^r_{ji}=Q^{r+1}_{ji}\circ\mathrm{d}$, $r=0,1$, where the exterior derivative $\mathrm{d}$ represents the surface gradient ($r=0$) or surface curl ($r=1$), respectively ({\em commuting property}). This way, the structure of the de Rham complex is preserved when projecting from one mesh to the other.
\subsection{Strong coupling: interpolation methods}
In a straightforward approach the interface data is interpolated by the Whitney map on one mesh and then integrated on the other mesh with the de Rham map. Since both Whitney and de Rham maps commute with the exterior derivative, this approach enjoys the commuting property. The implementation is relatively straightforward, and the numerical effort grows linearly with the number of interface degrees of freedom. However, in general the method exhibits poor convergence of the approximation property, when compared to convergence rates that can be achieved by finite element methods in the adjacent domains \cite[p.~292]{Canuto2007}, \cite[Sect.~1.4]{Flemisch2007}. This has been confirmed in \cite[Fig.~6]{Shi2008} for nodal elements, where a mortar element method was used as baseline. On the other hand, numerical experiments in \cite{Journeaux2014} indicate that this problem is less pronounced for edge and facet elements.
\subsection{Weak coupling: projection methods}\label{sec_projection}
Projection methods define the coupling operators by orthogonalizing the residual with respect to a Lagrange multiplier space. For nodal elements the space $\mathcal{W}^0$ is usually chosen for this purpose.\par
Consider finite element formulations in domains $\Omega_{1,2}$, to be coupled via the interface $\Gamma_1=-\Gamma_2$. The weak continuity condition characterized by $Q^r_{ji}$ can be directly taken into account as a constraint in the construction of the global variational space. This yields a symmetric positive definite finite element system. Alternatively, one may work with the unconstrained space and take into account the weak continuity condition by Lagrange multipliers. This yields a symmetric indefinite saddle point problem. In \cite{Lange2012} it is explained how the symmetric positive definite formulation can be restored on the algebraic level, by eliminating the Lagrange multipliers algebraically.\par
Projection methods require the inversion of a mass matrix. This can be done efficiently by using biorthogonal bases for the Lagrange multiplier space \cite{Lange2012,Wohlmuth2002}, resulting in a diagonal mass matrix. Nevertheless, projection methods are demanding to implement since discontinuous functions have to be numerically integrated. Strategies to cope with this are presented in \cite[Fig.~9]{Bouillault2003}, \cite{Gander2013} and \cite[Sect.~3]{Journeaux2013a}.\par
For completeness, we also mention Nitsche-type mortaring, where an additional independent function space is introduced on the interface. This approach has been recently generalized to Maxwell's equations \cite{Hollaus2010}.
\begin{remark}\label{rem_mortar}
In our model problem the interface is a manifold without boundary. Not so if the interface intersects the exterior boundary or in case of several sub-domains. The definition of the Lagrange multiplier space becomes more involved then, even for nodal elements. It cannot be chosen as one of the trace spaces, but only as a subspace of it, see \cite{Bernardi2005} for details.
\end{remark}
\subsection{From nodal elements to edge elements}
The generalization of projection-based methods from nodal to edge elements is not obvious. In particular, only a few references aim at a rigorous theoretical analysis of such mortar element methods \cite{Belgacem2001,Hu2008}. Many authors again use the finite element space $\mathcal{W}^1$ as the Lagrange multiplier space, e.g.~\cite{Bouillault2003,Buffa2003c,Hoppe1999}. Unfortunately, this policy does not yield a structure-preserving discretization. As a remedy, a div-conforming space rather than the curl-conforming space $\mathcal{W}^1$ should be used as Lagrange multiplier space. An obvious choice is the lowest order Raviart-Thomas space, which is, in fact, the space of div-conforming Whitney forms on a 2D manifold. However, this does not yield a stable discretization of the $L^2$ inner product \cite[Sect.~1]{Buffa2007}. In practice, this means that the mass matrix may become (nearly) singular. Consequently, non-existence of biorthogonal bases has been proven in \cite[Sect.~VI]{Lange2012}.\par
This problem can be avoided by using a so-called {Buffa-Christiansen (B-C) space} as Lagrange multiplier space \cite{Buffa2007}, \cite[Sect.~4]{Smirnova2013}. B-C spaces consist of certain subspaces of Whitney forms that are defined on the barycentric refinement. B-C spaces form a complex, which is dual to the Whitney complex, featuring a stable discretization of the duality pairing. So far, the B-C complex has been successfully applied to establish a multiplicative Calder\'on preconditioner for the Electric Field Integral Equation, resulting in a dramatic speedup of the iterative convergence \cite[Fig.~10]{Andriulli2008}. Application to nonconforming mesh coupling is a new proposal, according to our best knowledge.\par
The rest of the paper is organized as follows. In Sect.~\ref{sec_BC_main} we present a systematic construction of the B-C complex. After fixing definitions and notations in Sect.~\ref{sec_BC_notation}, the complex is constructed in Sect.~\ref{sec_BC_construction} and further characterized in Sect.~\ref{sec_BC_properties} by its main properties. Sect.~\ref{sec_open_interface} generalizes the construction to a B-C complex over a manifold with boundary. The mesh coupling operators $Q^r_{ji}$ are presented and their properties discussed in Sect.~\ref{sec_coupling_operators}. Sect.~\ref{sec_experiment} is devoted to a numerical experiment, where we compare simple interpolation and two different projection-based methods to each other. 
\section{The Buffa-Christiansen complex}\label{sec_BC_main}
\subsection{Definitions and notation}\label{sec_BC_notation}
Denote by $\mathcal{T}_i$ the restriction of the finite element mesh in domain $\Omega_i$ to the coupling interface of dimension $n=2$, with ordered sets of nodes $\mathcal{T}_i^0$, edges $\mathcal{T}_i^1$ and facets $\mathcal{T}_i^2$, $i=1,2$. In the sequel, we suppress the domain indices `$1$' and `$2$', since the same construction applies to either of the interface meshes. The sets $(\mathcal{T}^r,\partial)$, $r=0,1,2$, form a simplicial complex.\footnote{We do not distinguish cells in $\mathcal{T}^r$ from their geometric realizations.} The corresponding Whitney complex is $(\mathcal{W}^r,\mathrm{d})$, $r=0,1,2$, with standard bases $\boldsymbol{\lambda}^r=(\boldsymbol{\lambda}^r_t)$, indexed by $t\in\mathcal{T}^r$. The barycentric refinement of $\mathcal{T}$ is $\widetilde{\mathcal{T}}$, with corresponding Whitney complex $(\widetilde{\mathcal{W}}^r,\mathrm{d})$, and standard bases $\tilde{\boldsymbol{\lambda}}^r=(\tilde{\boldsymbol{\lambda}}^r_w)$, $w\in\widetilde{\mathcal{T}}^r$. The B-C complex is a subcomplex $(\mathcal{B}^r,\mathrm{d})\subset(\widetilde{\mathcal{W}}^r,\mathrm{d})$.\par
Let $\mathcal{V}$ be the barycentric dual of $\mathcal{T}$.
\begin{enumerate}
\item Cells $v\in\mathcal{V}^q$ are in one-to-one correspondence with cells $t\in\mathcal{T}^r$, $r+q=2$.
%, where $n$ is the dimension of the interface.
We write
\begin{equation}\label{defstar}
\star:\mathcal{T}^r\xrightarrow{\sim}\mathcal{V}^q:t\mapsto v\,.
\end{equation}
Sets $\mathcal{V}^q$ are considered ordered, with the order induced from $\mathcal{T}^r$ by $\star$. We fix the inner orientations in the dual mesh by requiring that $(\star t,t)$ is in the orientation of the interface.
\item Cells $v\in\mathcal{V}^q$ can be expressed as formal linear combinations (chains) of cells $w\in\widetilde{\mathcal{T}}^q$,
\begin{equation}\label{chain}
v=\sum_{w\in\widetilde{\mathcal{T}}^q}c_{vw}^qw,\quad c_{vw}^q\in\{-1;0;1\}\,.
\end{equation}
\item We will construct bases $\boldsymbol{\mu}^q=(\boldsymbol{\mu}^q_v)$ for $\mathcal{B}^q$, indexed by $v\in\mathcal{V}^q$, by linear combination,
\begin{equation}\label{lincomb}
\boldsymbol{\mu}^q_v=\sum_{w\in\widetilde{\mathcal{T}}^q}R_{vw}^q\tilde{\boldsymbol{\lambda}}^q_w,\quad v\in\mathcal{V}^q\,.
\end{equation}
Once the coefficients $R_{vw}^q\in\mathbb{R}$ are fixed we let $\mathcal{B}^q=\mathrm{span}\,\boldsymbol{\mu}^q$.
\end{enumerate}
Roughly speaking, the B-C complex mimics the properties exhibited by the Whitney complex. In addition, the spaces $\mathcal{W}^r$ and $\mathcal{B}^q$ are in stable duality (Prop.~\ref{prop_infsup} below).
\subsection{Construction of the B-C complex}\label{sec_BC_construction}
The construction of the B-C complex boils down to selecting the coefficients $R_{vw}^q$ in \eqref{lincomb}. In the sequel, we will establish a set of minimal assumptions which uniquely determine $R_{vw}^q$.
%in connection with Props.~\ref{prop_intpol} and \ref{prop_extder}.
The resulting B-C spaces agree with those communicated in \cite{Buffa2007} where the treatment relies on classical vector analysis. However, the differential geometric framework adopted here makes it easier to separate between topological and geometric aspects of the construction. In particular, we will see that the construction relies solely on the topology of the mesh and can be done without invoking a metric.\par
Let
\[
\Gamma_h=\bigcup\limits_{t\in\mathcal{T}^2}\overline{t}
\]
be the discretized coupling interface, where $\overline{t}$ is the closure of $t$. We assume that the interface $\Gamma_h$ has no boundary, $\partial\Gamma_h=0$. This corresponds to the situation depicted in Fig.~\ref{fig_Model_Onion}. The extension to interfaces with boundary is treated in Sect.~\ref{sec_open_interface}.\par
\begin{assumption}\label{ass_support}
%%% Support (on a topological space X) is defined as the closure of the subset of X where the function is nonzero.
(Local support) The support of the B-C basis forms is defined recursively. Let $U^q_v=\mathrm{supp}(\boldsymbol{\mu}^q_v)$, $v\in\mathcal{V}^q$, $0\le q\le 2$. We let $U^2_v=\overline{v}$, and
\[
U_v^q=\!\!\!\!\!\!\bigcup_{\{u\in\mathcal{V}^{q+1}\,|\,v\in\partial u\}}\!\!\!\!\!\!U_u^{q+1}\,,\quad 0\le q<n\,.
\]
This yields the regions shaded in gray in \cite[Figs.~1--3]{Buffa2007}. We also require that the basis forms have zero trace on the boundary of their support, $\mathrm{t}\boldsymbol{\mu}^q_v=0$ on $\partial U^q_v$.
\end{assumption}
\noindent Assumption \ref{ass_support} is implemented by selecting $R_{vw}^q=0$ for $w\not\subset U^q_v$, and for $w\subset\partial U^q_v$.
\begin{assumption}\label{ass_sameness}
All cells $w\in\widetilde{\mathcal{T}}^q$ that are contained in the dual cell $v\in\mathcal{V}^q$ shall contribute with the same weight to the basis form $\boldsymbol{\mu}^q_v$.
\end{assumption}
\noindent The number of cells contained in the dual cell $v$ is given by
\begin{equation}\label{incidentcells}
n_v=\sum_{w\in\widetilde{\mathcal{T}}^q}|c^q_{vw}|\,,
\end{equation}
and we let
\begin{equation}\label{samecoef}
R_{vw}^q=\frac{1}{n_v}c_{vw}^q\mbox{ for }c_{vw}^q\ne 0\,,v\in\mathcal{V}^q,w\in\widetilde{\mathcal{T}}^q\,,
\end{equation}
with the coefficients $c_{vw}^q$ from \eqref{chain}.%\par
\begin{assumption}\label{ass_stability}
Consider the dual cell $v\in\mathcal{V}^q$, and primal cell $t\in\mathcal{T}^r$, such that $v=\star t$, in the special case $r=q=1$. Cells $w\in\widetilde{\mathcal{T}}^q$ that are contained in the primal cell $t$ do not contribute to the basis form $\boldsymbol{\mu}^q_v$, i.e., $R_{vw}^q=0$.
\end{assumption}
\noindent %An example is depicted in Fig.~\ref{fig_BC_1_Forms}. The pieces of the primal edge $t$ do not contribute to the B-C $1$-form associated with dual edge $\star t$. 
This choice is essential for a stable duality.\par%\bigskip
The next proposition is an immediate consequence of Assumptions \ref{ass_support} and \ref{ass_sameness}.
\begin{proposition}\label{prop_intpol}
(Interpolation Property)
It holds that
\[
\int_u\boldsymbol{\mu}^q_v=\delta_{uv},\quad u,v\in\mathcal{V}^q,\boldsymbol{\mu}^q_v\in\mathcal{B}^q\,.
\]
The degrees of freedom are given by the de Rham maps, that is, integrals over dual nodes, edges and facets, respectively.
\end{proposition}
\begin{proof}
Consider $u,v\in\mathcal{V}^q$. For $u\ne v$ it follows from Assumption \ref{ass_support} that either $u\not\subset U^q_v$ or $u\subset\partial U^q_v$, hence
\[
\int_u\boldsymbol{\mu}^q_v=0\,.
\]
For $u=v$ we receive
\[
\int_u\boldsymbol{\mu}^q_v=\sum_{w\in\widetilde{T}^q}c_{vw}^qR_{vw}^q=1\,,
\]
where we used \eqref{chain}, \eqref{lincomb} and \eqref{samecoef}.
\end{proof}
The remaining undetermined coefficients $R_{v\genindex}^q$ related to the interior of the support $U_v^q$ are chosen in such a way that the following proposition holds. We first state the proposition and then derive the construction principle.
\begin{proposition}\label{prop_extder}
(Discrete Exterior Derivative) The matrix of the exterior derivative in the bases $\boldsymbol{\mu}^q$ is the transpose of the matrix of the exterior derivative in the standard bases $\boldsymbol{\lambda}^r$, up to sign,\footnote{From the geometric viewpoint, the dual grid should be outer oriented, and the B-C complex should be twisted. However, to keep things simple, we consider the dual grid inner oriented, and the B-C complex ordinary. This yields a sign correction, depending on $r$ and $n$, whenever the incidence matrix is applied on the dual side. The factor $(-1)^r$ holds for $n=2$, and for the orientation convention taken for the dual mesh.}
\[
\mathrm{d}\boldsymbol{\lambda}^r_s=\sum_{t\in\mathcal{T}^{r+1}}D_{st}^r\boldsymbol{\lambda}^{r+1}_t\mbox{\quad implies\quad}\mathrm{d}\boldsymbol{\mu}^{q-1}_{\star t}=(-1)^r\sum_{s\in\mathcal{T}^r}D_{st}^r\boldsymbol{\mu}^q_{\star s}\,.
\]
\end{proposition}
\begin{proof}
By construction. The above equation can be recognized as a relation in $\widetilde{\mathcal{W}}^q$, and therefore be re-written as
\begin{subequations}
\begin{equation}\label{fixcoeffsa}
\int_z\mathrm{d}\boldsymbol{\mu}^{q-1}_{\star t}=(-1)^r\int_z\sum_{s\in\mathcal{T}^r}D_{st}^r\boldsymbol{\mu}^q_{\star s}\,,\quad t\in\mathcal{T}^{r+1},\forall z\in\widetilde{\mathcal{T}}^q: z\subset U^{q-1}_{\star t}.
\end{equation}
Invoking the basis expansion \eqref{lincomb} and using Prop.~\ref{prop_intpol} yields the equivalent expression
\begin{equation}\label{fixcoeffsb}
\sum_{\substack{w\in\widetilde{\mathcal{T}}^{q-1}:\\ w\subset\partial z}}\!\!\!\!\!\!R_{\star tw}^{q-1}\widetilde{D}_{wz}^{q-1}
=(-1)^r\sum_{\substack{s\in\mathcal{T}^r:\\ s\subset\partial t}}\!\!\!D_{st}^rR_{\star sz}^q\,,\quad t\in\mathcal{T}^{r+1},\forall z\in\widetilde{\mathcal{T}}^q: z\subset U^{q-1}_{\star t}\,.
\end{equation}
\end{subequations}
The equation \eqref{fixcoeffsb} allows to recursively determine the coefficients $R_{\genindex\genindex}^{q-1}$ from $R_{\genindex\genindex}^q$.
\end{proof}
\subsubsection{B-C 2-forms}\label{BC_2-Forms}
The space $\mathcal{B}^2$ is uniquely determined by Assumptions \ref{ass_support} -- \ref{ass_sameness}. The coefficients $R_{\genindex\genindex}^2$ are depicted in \cite[Fig.~3]{Buffa2007}.
\subsubsection{B-C 1-forms}($r=0,q=2$)\label{B-C_1-Forms}
%
%\begin{figure}
%\centering
%%\includegraphics[width=0.7\linewidth]{BC_1_Forms}
%\caption{\label{fig_BC_1_Forms}Construction of the B-C $1$-form associated with dual edge $\star t$. The B-C $1$-form is supported in $U^1_{\star t}$, the region shaded in gray. Node $s$ is in the boundary of the primal edge $t$, and $\star s$ is its dual Voronoi cell. The coefficients $R_{\star t\genindex}^1\in\{0;\pm 1/2\}$ related to the thick edges are determined by Assumptions \ref{ass_support} -- \ref{ass_stability}. The remaining coefficients can be obtained by evaluating \eqref{fixcoeffsb} successively for the numbered facets. Small arrows indicate orientations of edges in the barycentric refinement. The figure is drawn on top of \cite[Fig.~2]{Buffa2007}.}
%\end{figure}
%
A typical configuration %is shown in Fig.~\ref{fig_BC_1_Forms}, compare 
follows that of \cite[Fig.~2]{Buffa2007}.
%Consider the Voronoi cell $\star s$. 
The edge coefficients $R_{\star t\genindex}^1\in\{0;\pm 1/2\}$ %related to the thick edges 
are determined by Assumptions \ref{ass_support} -- \ref{ass_stability}. The remaining coefficients can be obtained by evaluating \eqref{fixcoeffsb} successively for the %numbered
facets. % $z$ in $U^1_{\star t}$.
The sum on the right hand side of \eqref{fixcoeffsb} reduces to a single term, since each facet is contained in the support of exactly one B-C basis $2$-form. %For example, for the facet $z$ in Fig.~\ref{fig_BC_1_Forms} we receive $0+1/2-R_{\star tw}^1=R_{\star sz}^2=+1/12$, that is, $R_{\star tw}^1=5/12$. This agrees with the result reported in \cite[Fig.~2]{Buffa2007}.
%\par
The system \eqref{fixcoeffsb} is consistent, despite that the number of equations exceeds the number of unknowns by one.\par
\begin{figure}
\centering
\includegraphics[width=0.3\textwidth]{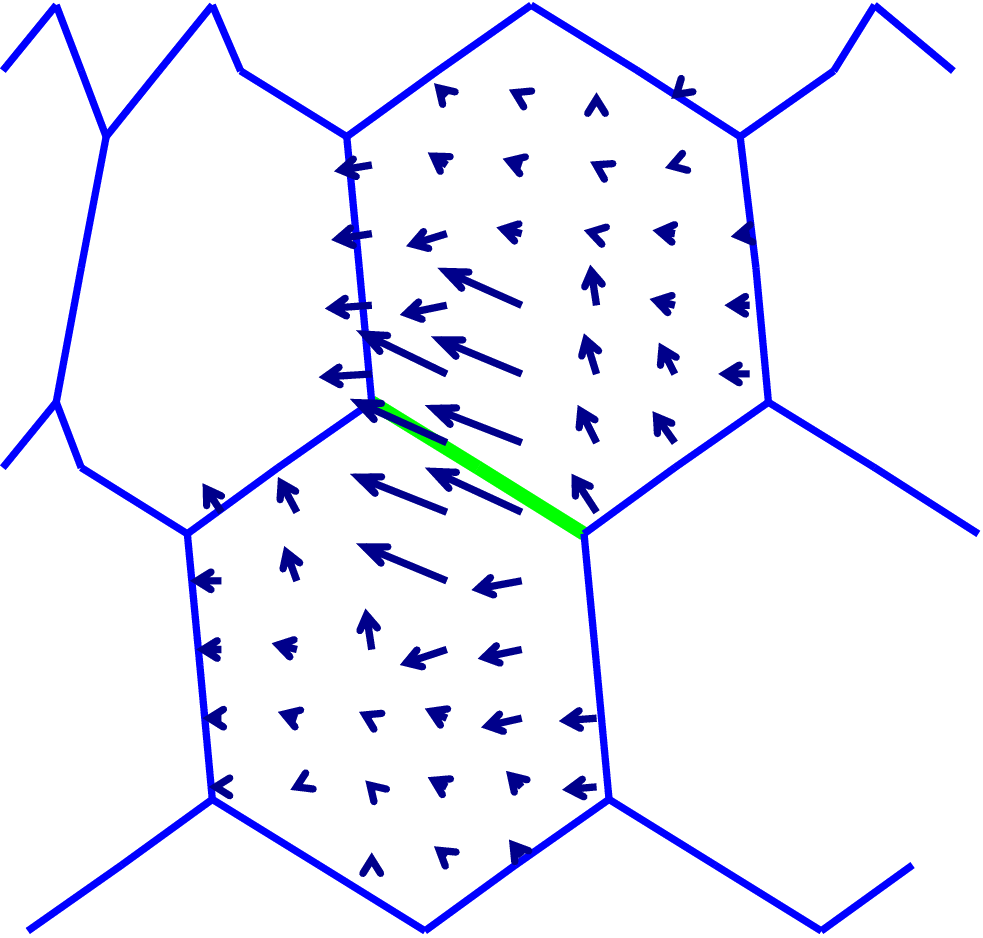}
\caption{\label{fig_Plot_BC_1_Form}Euclidean vector proxy of the B-C basis $1$-form associated with the center edge.}
\end{figure}
The Euclidean vector proxy of a B-C basis $1$-form is plotted in Fig.~\ref{fig_Plot_BC_1_Form}.
\subsubsection{B-C 0-forms}($r=1,q=1$)\label{BC_0-Forms}
%
%\begin{figure}
%\centering
%\includegraphics[width=0.7\linewidth]{BC_0_Forms}
%\caption{\label{fig_BC_0_Forms}Construction of the B-C $0$-form associated with dual node $\star t$. The B-C $0$-form is supported in $U^0_{\star t}$, the region shaded in gray. The coefficients $R_{\star t\genindex}^0\in\{0;1\}$ of the thick nodes are determined by Assumptions \ref{ass_support} -- \ref{ass_sameness}. The remaining coefficients can be obtained by evaluating \eqref{fixcoeffsb} for the edges of the barycentric refinement. The dashed line is a spanning tree of the edge graph. The figure is drawn on top of \cite[Fig.~1]{Buffa2007}.}
%\end{figure}
%
Consider the configuration depicted in %Fig.~\ref{fig_BC_0_Forms}, compare
\cite[Fig.~1]{Buffa2007}. The B-C $0$-form is supported in %$U^0_{\star t}$, 
the region shaded in gray. The coefficients $R_{\star t\genindex}^0\in\{0;1\}$ of the %thick 
center node and the boundary nodes are determined by Assumptions \ref{ass_support} -- \ref{ass_sameness}. The remaining coefficients can be obtained by evaluating \eqref{fixcoeffsb} for the edges %$z$
contained in the primal facet, %$t$
the center triangle. %in Fig.~\ref{fig_BC_0_Forms}. 
This case is slightly more complicated than before, because edges that connect a vertex to the barycenter of the triangle are contained in the support of two B-C basis $1$-forms. They give rise to two terms in the right hand side of \eqref{fixcoeffsb}. %For example, for the edge $z$ in Fig.~\ref{fig_BC_0_Forms} we obtain $+1-R_{\star tw}^0=-(+R_{\star s_1z}^1-R_{\star s_2z}^1)= 4/5$, that is, $R_{\star tw}^0=1/5$. Again, this agrees with the result reported in \cite[Fig.~1]{Buffa2007}.
\par
Equation \eqref{fixcoeffsa} must hold for all edges of the barycentric refinement, not just for those contained in the central triangle. To verify this we first show that the $q$-form on the right hand side is closed. By making use of Prop.~\ref{prop_extder} recursively we find
\begin{align*}
\mathrm{d}\sum_{s\in\mathcal{T}^r}D_{st}^r\boldsymbol{\mu}^q_{\star s}
&=\sum_{s\in\mathcal{T}^r}D_{st}^r\mathrm{d}\boldsymbol{\mu}^q_{\star s}\\
&=(-1)^{r-1}\sum_{s\in\mathcal{T}^r}\sum_{s^\prime\in\mathcal{T}^{r-1}}D_{s^\prime s}^{r-1}D_{st}^r\boldsymbol{\mu}^{q+1}_{\star s^\prime}=0\,.
\xqedhere{54pt}{\qed}
\end{align*}
\par
We therefore invoke a tree-cotree decomposition of the edge graph \cite[Def.~5.2]{bossavit}. %A suitable tree is indicated in Fig.~\ref{fig_BC_0_Forms} by the dashed line. 
Equation \eqref{fixcoeffsa} calls for verification for the spanning tree edges outside the central triangle. It is easy to see that this is consistent with the prescribed zero coefficients related to the nodes in the boundary of the support. %$U^0_{\star t}$.
\begin{remark}\label{rem_nometric}(Metric-free Property)
Observe that no metric information entered the construction of the B-C complex. Therefore, the coefficients $R_{\genindex\genindex}^{\genindex}$ depend only on the topology but not on the geometry of the mesh.
\end{remark}
\begin{remark}(Dimensions other than two)
If the above construction principle is applied to the one-dimensional interface ($n=1$) between two-dimensional domains, then, as a result, the B-C complex on the barycentric dual mesh turns out to be the Whitney complex.\par
On the other hand, in three dimensions ($n=3$), the construction principle does not uniquely fix the B-C complex.

\end{remark}
\begin{remark}(Reconstruction of the Whitney complex)
If the roles of primal and dual grids are interchanged, and the above construction principle is applied, then the Whitney complex on $\mathcal{T}$ is reconstructed as a subcomplex of the Whitney complex on the barycentric refinement $\widetilde{\mathcal{T}}$. This approach yields the coefficients reported in \cite[eq.~(42)]{Andriulli2008}.
\end{remark}
\subsection{Main properties of the B-C complex}\label{sec_BC_properties}
The definitions in Sect.~\ref{sec_BC_notation} and Assumptions \ref{ass_support} -- \ref{ass_stability} can be used to infer further characteristics of the B-C complex besides the Interpolation and Discrete Exterior Derivative properties in Props.~\ref{prop_intpol} and \ref{prop_extder}, and the Metric-free Property of Remark~\ref{rem_nometric}.
\begin{proposition}\label{prop_condis}
(Conforming Discretization of Trace Spaces). It holds that
\[
\mathcal{B}^q\subset H^{-1/2}_\perp\mathsf{\Lambda}^q(\mathrm{d},\Gamma_h)\,.
\]
\end{proposition}
\begin{proof}
This property is inherited from the spaces $\widetilde{\mathcal{W}}^q$ through \eqref{lincomb}.
\end{proof}
\noindent Note that the spaces $H^{-1/2}_\perp\mathsf{\Lambda}^q(\mathrm{d},\Gamma_h)$ contain traces of differential forms whose exterior derivatives are square-integrable. For a proper definition see \cite[p.~23]{Kurz2012}. They encompass and generalize the well-known spaces $H^{-1/2}(\Gamma_h)$, $\boldsymbol{H}^{-1/2}_\perp(\mathrm{curl}_\Gamma,\Gamma_h)$, and $L^2(\Gamma_h)$, respectively. Functions in $\mathcal{B}^0$ are continuous, differential $1$-forms in $\mathcal{B}^1$ have a continuous (tangential) trace across interfaces,\footnote{curl-conforming, in the language of vector analysis. Note that in \cite{Buffa2007} the space $\hodge\mathcal{B}^1$ is considered, where $\hodge$ is the Hodge operator on $\Gamma_h$ induced by the Euclidean metric. This corresponds to a rotation by $\pi/2$, hence div-conforming fields. Remark \ref{rem_pairing} provides a rationale for our approach.} while differential $2$-forms in $\mathcal{B}^2$ are discontinuous.
\begin{proposition}\label{prop_complex}
(Complex Property). The B-C spaces form a discrete de Rham complex,
\[
0\rightarrow\mathcal{B}^0\xrightarrow{\mathrm{d}}\mathcal{B}^1\xrightarrow{\mathrm{d}}\mathcal{B}^2\xrightarrow{\mathrm{d}} 0\,.
\]
The dimensions of its cohomology groups are given by the Betti numbers related to $\Gamma_h$. They are relevant in case of topologically non-trivial situations.
\end{proposition}
\begin{proof}
Consequence of Prop.~\ref{prop_extder}.
\end{proof}
\begin{proposition}\label{prop_comproj}
The projectors $\Pi$ from sufficiently regular subspaces of the spaces $H^{-1/2}_\perp\mathsf{\Lambda}^q(\mathrm{d},\Gamma_h)$ to $\mathcal{B}^q$ commute with the exterior derivative,
\[
\Pi\circ\mathrm{d}=\mathrm{d}\circ\Pi.
\]
\end{proposition}
\begin{proof}
Recall, the exterior derivative is represented by an incidence matrix $D_{st}^r$, in a given simplicial complex. Therefore it holds $\partial\star s=(-1)^r\sum_{t\in\mathcal{T}^{r+1}}D_{st}^r\star t$. Let $\boldsymbol{\beta}\in H^{-1/2}_\perp\mathsf{\Lambda}^{q-1}(\mathrm{d},\Gamma_h)$ such that the de Rham maps of $\boldsymbol{\beta}$ and $\mathrm{d}\boldsymbol{\beta}$ exist. Then
\begin{align*}
\Pi\,\mathrm{d}\,\boldsymbol{\beta}&=\sum_{s\in\mathcal{T}^r}\boldsymbol{\mu}^q_{\star s}\int_{\star s}\mathrm{d}\boldsymbol{\beta}
=\sum_{s\in\mathcal{T}^r}\boldsymbol{\mu}^q_{\star s}\int_{\partial\star s}\boldsymbol{\beta}\\
&=\sum_{t\in\mathcal{T}^{r+1}}\left((-1)^r\sum_{s\in\mathcal{T}^r}D_{st}^r\boldsymbol{\mu}^q_{\star s}\right)\int_{\star t}\boldsymbol{\beta}
=\sum_{t\in\mathcal{T}^{r+1}}\mathrm{d}\boldsymbol{\mu}^{q-1}_{\star t}\int_{\star t}\boldsymbol{\beta}\\
&=\mathrm{d}\sum_{t\in\mathcal{T}^{r+1}}\boldsymbol{\mu}^{q-1}_{\star t}\int_{\star t}\boldsymbol{\beta}=\mathrm{d}\,\Pi\,\boldsymbol{\beta}
\end{align*}
holds, where we used Stokes' Theorem and Prop.~\ref{prop_extder}.
\end{proof}
\begin{proposition}\label{prop_unity}
(Partition of Unity) The functions $\boldsymbol{\mu}^0_v\in\mathcal{B}^0$, $v\in\mathcal{V}^0$, form a partition of unity.\footnote{We assume that all facets in $\mathcal{T}^2$ are in the orientation of the interface. Then our convention induces the orientation $+1$ for all dual nodes. Otherwise, the orientation of the dual nodes has to be taken into account when forming the partition of unity.}
\end{proposition}
\begin{proof}
Consider that
\[
\mathrm{d}\sum_{t\in\mathcal{T}^2}\boldsymbol{\mu}_{\star t}^0=-\sum_{s\in\mathcal{T}^1}\left(\sum_{t\in\mathcal{T}^2}D_{st}^1\right)\boldsymbol{\mu}_{\star s}^1=0
\]
holds, where we used Prop.~\ref{prop_extder}. The expression in parenthesis is zero, since each edge $s$ is incident with two facets $t$, with different signs, provided that boundary edges have been removed from $\mathcal{T}^1$ by appropriate boundary conditions \cite[Sect.~4.2]{Buffa2007}. We infer that the sum on the left side must be constant in each connected component of the interface. From Prop.~\ref{prop_intpol} it follows that the constant is one, which completes the proof. This demonstrates that the Partition of Unity property is a corollary to Props.~\ref{prop_intpol} and \ref{prop_extder}.
\end{proof}
\begin{proposition}\label{prop_voronoi}
(Euclidean Volume Form) Denote by $\hodge$ the Hodge operator on $\Gamma_h$ induced by the Euclidean metric. The B-C $2$-form $\boldsymbol{\mu}^2_v\in\mathcal{B}^2$ is supported in the Voronoi cell $v\in\mathcal{V}^2$. It is a fixed multiple of the volume form $\hodge 1$ in $v$.
\end{proposition}
\begin{proof}
By construction, see Assumption \ref{ass_sameness} above.
\end{proof}
\begin{remark}
The link to the Euclidean volume form is inherited from the Whitney complex.
\end{remark}
\begin{proposition}\label{prop_infsup}
(Stable Duality) Consider a quasi-uniform family of meshes $(\mathcal{T}_h)$ in the coupling interface, with associated B-C and Whitney complexes $(\mathcal{B}_h^q,\mathrm{d})$ and $(\mathcal{W}_h^r,\mathrm{d})$, respectively. The pairing 
\[
\mathcal{B}_h^q\times\mathcal{W}_h^r\to\mathbb{R}:(\boldsymbol{\beta},\boldsymbol{\omega})\mapsto b(\boldsymbol{\beta},\boldsymbol{\omega})=\int_{\Gamma_h}\boldsymbol{\beta}\wedge\boldsymbol{\omega}
\]
is nondegenerate. It satisfies a discrete inf-sup condition uniformly in $h$,
\[
\inf_{\boldsymbol{\omega}\in\mathcal{W}^r_h}\sup_{\boldsymbol{\beta}\in\mathcal{B}^q_h}\frac{b(\boldsymbol{\beta},\boldsymbol{\omega})}{||\boldsymbol{\beta}||\,||\boldsymbol{\omega}||}\ge 1/C_1\,,
\]
where the norms are the graph norms of the trace spaces, $||\cdot||=||\cdot||_{H^{-1/2}_\perp\mathsf{\Lambda}(\mathrm{d},\Gamma_h)}$, \cite[p.~23]{Kurz2012}.
\end{proposition}
\begin{proof}
See \cite[Sect.~3.3]{Buffa2007}, in particular their Props.~3.12-3.14. Some mild local non-degeneracy condition is required.
\end{proof}
\begin{remark}\label{rem_pairing}
The pairing $b(\cdot,\cdot)$ is independent of metric. It is related to the metric-dependent $L^2$ inner product $(\cdot,\cdot)$ by
$b(\boldsymbol{\beta},\boldsymbol{\omega})=(\hodge\boldsymbol{\beta},\boldsymbol{\omega})$\,.
\end{remark}
\begin{remark}
The approximation properties of the B-C forms are discussed in \cite[Sect.~3.1]{Buffa2007}.
\end{remark}
\subsection{Extension to interfaces with boundary}\label{sec_open_interface}
The construction of appropriate Lagrange multiplier spaces for mortar element methods including interfaces with boundaries can be an involved task, compare Remark \ref{rem_mortar}. We therefore skip the general case and restrict ourselves to an illustrative example, to highlight the treatment of boundaries in the context of B-C spaces.\par
%
%\begin{figure}
%\centering
%\includegraphics[width=0.37\linewidth]{BC_1_Forms_Boundary}\hspace{1cm}
%\includegraphics[width=0.37\linewidth]{BC_0_Forms_Boundary}
%\caption{\label{fig_BC_Forms_Boundary}Construction of B-C forms for an interface with boundary $\partial\Gamma_h$. %\\[0.5\baselineskip]
%Left: Construction of the B-C $1$-form associated with dual edge $\star t$. The B-C $1$-form is supported in $U^1_{\star t}$, the region shaded in gray. Node $s$ is in the intersection of primal edge $t$ with the boundary. Only the dual Voronoi cell $\star s$ is considered here. The coefficients $R_{\star t\genindex}^1$ are obtained in the same way as in Fig.~\ref{fig_BC_1_Forms}, except for Assumption \ref{ass_stability}, which is discarded. The resulting coefficients for the dual edges emanating from node $s$ in counterclockwise order are $0;-1/6;-1/3;-1/2;-1/6;1/6;0$. %\\[0.5\baselineskip]
%Right: Construction of the B-C $0$-form associated with dual node $\star t$. The B-C $0$-form is supported in $U^0_{\star t}$, the region shaded in gray. The dark gray part interacts with the boundary. The coefficients related to nodes $w_1,w_2$ are computed with the same approach as in Fig.~\ref{fig_BC_0_Forms}. This yields $R_{\star tw_1}^0=R_{\star tw_2}^0=1/2$. %\\[0.5\baselineskip]
%The figures are drawn on top of \cite[Fig.~4]{Buffa2007}.}
%\end{figure}
%
The choice of the spaces depends on the intended use, whether a Dirichlet condition is to be imposed on $\mathcal{W}$ \cite[Sect.~IV]{Andriulli2008}, or not. The construction in \cite[Sect.~4.2]{Buffa2007} aims at a B-C complex $(\mathcal{B},\mathrm{d})$ in duality with the complex $(\mathcal{W}_0,\mathrm{d})$. In contrast, we will construct a B-C complex $(\mathcal{B}_0,\mathrm{d})$ with vanishing boundary trace, in duality with the complex $(\mathcal{W},\mathrm{d})$.\par
The dual mesh $\mathcal{V}$ is truncated at the boundary $\partial\Gamma_h$. Assumptions \ref{ass_support} and \ref{ass_sameness} stay intact. Hence no modification in the construction of B-C $2$-forms, see Sect.~\ref{BC_2-Forms}.\par
Consider the B-C $1$-form associated with a dual edge $\star t$, such that $\overline{t}\cap\partial\Gamma_h\ne\emptyset$. Either $\overline{t}$ is contained in the boundary, or it has one point in common with the boundary. The first case is treated identically to the interior case. A typical configuration for the second case is shown in % Fig.~\ref{fig_BC_Forms_Boundary} left, compare 
\cite[Fig.~4]{Buffa2007}. The construction runs along the same lines as in Sect.~\ref{B-C_1-Forms}, except for Assumption \ref{ass_stability}, which is discarded. %The nonzero coefficient related to edge $w\subset t$ is obtained as $R_{\star tw}^1=-1/6$.
\par
The construction of B-C $0$-forms follows Sect.~\ref{BC_0-Forms}, even in the part of the support that interacts with the boundary.%, see Fig.~\ref{fig_BC_Forms_Boundary} right.
\begin{remark}(Partition of Unity violation)
The Partition of Unity property is violated for the basis forms of $\mathcal{B}_0^0$, due to the boundary condition. The sum of the basis forms yields the constant function one in the interior of the domain. A linear decay to zero occurs in the outmost layer of elements $w\in\widetilde{\mathcal{T}}$.
\end{remark}
\begin{remark}
The case studied in \cite[Sect.~4.2]{Buffa2007} can be tackled by our construction principles as well, with a slight modification of Assumption \ref{ass_stability}.
\end{remark}
\section{Projection-based mesh coupling operators}\label{sec_coupling_operators}
\subsection{Definition of mesh coupling operators}\label{sec_def_operators}
For $ji\in\{12;21\}$ the projection-based mesh coupling operators $Q^r_{ji}:\mathcal{W}^r_i\to\mathcal{W}^r_j:\boldsymbol{\omega}_i\mapsto\boldsymbol{\omega}_j$ are defined by
\begin{equation}\label{defop}
b(\boldsymbol{\beta}_j,\boldsymbol{\omega}_j)=b(\boldsymbol{\beta}_j,\boldsymbol{\omega}_i)\quad\forall\boldsymbol{\beta}_j\in\mathcal{M}^q_j\,,
\end{equation}
where $\mathcal{M}^q_j$ is the discrete Lagrange multiplier space, and the pairing $b(\cdot,\cdot)$ is defined in Prop.~\ref{prop_infsup}. In particular, we pick $\mathcal{M}^q_j=\mathcal{B}^q_j$ to receive $Q_{ji}^{r,\textnormal{B-C}}$. As discussed in Sect.~\ref{sec_projection}, $\mathcal{M}^q_j=\mathcal{W}^q_j$ is not an option, due to lack of stability.\par
For comparison we also consider the Galerkin case $\mathcal{M}^q_j=\hodge^{-1}\mathcal{W}^r_j$. In the light of Remark \ref{rem_pairing}, this choice corresponds to $Q_{ji}^{r,\textnormal{Galerkin}}$.
\begin{remark}\label{rem_extension} (Extension to trace spaces)
The pairing $b(\cdot,\cdot):\mathcal{B}_h^q\times\mathcal{W}_h^r\to\mathbb{R}$ extends continuously to a non-degenerate pairing $H^{-1/2}_\perp\mathsf{\Lambda}^q(\mathrm{d},\Gamma_h)\times H^{-1/2}_\perp\mathsf{\Lambda}^r(\mathrm{d},\Gamma_h)\to\mathbb{R}$. This has been exploited in \cite[Thm.~2]{Buffa2003d} for $r=q=1$, where the representation in terms of rotated Euclidean vector proxies reads
\[
\boldsymbol{H}^{-1/2}_\parallel(\mathrm{div}_\Gamma,\Gamma_h)\times\boldsymbol{H}^{-1/2}_\parallel(\mathrm{div}_\Gamma,\Gamma_h):(\boldsymbol{b},\boldsymbol{w})\mapsto\int_{\Gamma_h}(\boldsymbol{b}\times\boldsymbol{n})\cdot\boldsymbol{w}\,\mathrm{d}\Gamma\,.
\]
The situation is different in the Galerkin case. In general, {\em except for $r=0$}, the $L^2$ inner product {\em does not} extend to $H^{-1/2}_\perp\mathsf{\Lambda}^r(\mathrm{d},\Gamma_h)$. This observation renders the $Q_{ji}^{r,\textnormal{Galerkin}}$ dubious for $r>0$. For the same reason, the setting in the literature is frequently based on a {\em componentwise} duality between $H^{-1/2}(\Gamma_h)$ and $H^{1/2}(\Gamma_h)$, compare \cite[eq.~(2.12)]{Belgacem2001}, \cite{Hoppe1999}.
\end{remark}
By invoking the basis representations for $\boldsymbol{\beta}_j$, $\boldsymbol{\omega}_i$, and $\boldsymbol{\omega}_j$, \eqref{defop} yields the matrix equation
\begin{equation}\label{defopmatrix}
[M]_{jj}^r\{\omega_j\}=[M]_{ji}^r\{\omega_i\}\,,
\end{equation}
where vectors $\{\omega_i\}$ and $\{\omega_j\}$ collect the known and the unknown degrees of freedom, respectively. In particular,
\begin{align*}
[M]_{ji}^{r,\textnormal{B-C}}&=\int_{\Gamma_{j,h}}\{\boldsymbol{\mu}_j^q\}\wedge\{\boldsymbol{\lambda}_i^r\}^\top\,,\\
[M]_{ji}^{r,\textnormal{Galerkin}}&=\int_{\Gamma_{j,h}}\{\boldsymbol{\lambda}_j^r\}\cdot\{\boldsymbol{\lambda}_i^r\}^\top\,\mathrm{d}\Gamma\,,
\end{align*} 
with similar expressions for $i=j$.\par
The matrix $[M]_{jj}^{r,\textnormal{Galerkin}}$ is a symmetric positive definite mass matrix, while the matrix $[M]_{jj}^{r,\textnormal{B-C}}$ is an unsymmetric invertible matrix, with a very low condition number. Therefore we may formally solve \eqref{defopmatrix} for $\{\omega_j\}$, to obtain the matrix representation $[Q]_{ji}^r=\bigl([M]_{jj}^r\bigr)^{-1}[M]_{ji}^r$ of the mesh coupling operator. We will come back to the condition numbers in Sect.~\ref{sec_experiment}. 
\subsection{Properties of mesh coupling operators}
We assume that either the boundary trace of the B-C spaces or the boundary trace of the Whitney spaces vanishes, compare Sect.~\ref{sec_open_interface}. This prerequisite is trivially fulfilled for interfaces without boundary. We are now in the position to state the main result of the paper.
\begin{theorem}(Commuting Property)\label{thm_struct_preserv}
The mesh coupling operators $Q_{ji}^{r,\textnormal{B-C}}$ are commuting projectors of Whitney spaces, that is, the following diagram commutes.
\[
\begin{CD}
\boldsymbol{\omega}_i @>\displaystyle\mathrm{d}>> \mathrm{d}\boldsymbol{\omega}_i\\
@V\displaystyle Q_{ji}^{r,\textnormal{B-C}}VV @VV\displaystyle Q_{ji}^{r+1,\textnormal{B-C}}V\\[-0.3\baselineskip]
\boldsymbol{\omega}_j @>\displaystyle\mathrm{d}>> \mathrm{d}\boldsymbol{\omega}_j
\end{CD}
\]
\end{theorem}
\begin{proof}
From the definition \eqref{defop} of $Q_{ji}^{r,\textnormal{B-C}}$ it follows that
\[
b(\boldsymbol{\beta}_j,\boldsymbol{\omega}_j-\boldsymbol{\omega}_i)=0\quad\forall\boldsymbol{\beta}_j\in\mathcal{B}^q_j\,.
\]
Consider $\boldsymbol{\alpha}_j\in\mathcal{B}^{q-1}_j$. Thanks to Prop.~\ref{prop_complex} it holds that $\mathrm{d}\boldsymbol{\alpha}_j\in\mathcal{B}^q_j$. Hence
\[
b(\mathrm{d}\boldsymbol{\alpha}_j,\boldsymbol{\omega}_j-\boldsymbol{\omega}_i)=0\quad\forall\boldsymbol{\alpha}_j\in\mathcal{B}^{q-1}_j\,.
\]
By partial integration it follows that
\[
b(\boldsymbol{\alpha}_j,\mathrm{d}\boldsymbol{\omega}_j-\mathrm{d}\boldsymbol{\omega}_i)=0\quad\forall\boldsymbol{\alpha}_j\in\mathcal{B}^{q-1}_j\,.
\]
The above prerequisite ensures that there is no contribution due to the boundary. The term $\mathrm{d}\boldsymbol{\omega}_j$ is contained in the Whitney space $\mathcal{W}^{r+1}_j$, and is -- by definition -- identical to $Q_{ji}^{r+1,\textnormal{B-C}}\mathrm{d}\boldsymbol{\omega}_i$. This completes the proof.
\end{proof}
\begin{remark}
The operators $Q_{ji}^{r,\textnormal{Galerkin}}$ fail to commute with the exterior derivative, since $(\hodge\mathcal{W}^r,\mathrm{d})$ does not exhibit the complex property.
\end{remark}
\begin{proposition}\label{prop_opt}
(Quasi-Optimal Projection) The projection error is bounded by the best approximation error in the norm of the trace space,
\[
||\boldsymbol{\omega}_j-\boldsymbol{\omega}_i||\le C_1C_2\inf_{\boldsymbol{\eta}\in\mathcal{W}^r_j}||\boldsymbol{\eta}-\boldsymbol{\omega}_i||\,,
\]
where $||\cdot||$ and $C_1$ are defined in Prop.~\ref{prop_infsup}.
\end{proposition}
\begin{proof}
Consider the extension of $b(\cdot,\cdot)$ to the trace spaces from Remark \ref{rem_extension}. The proposition follows from the discrete inf-sup condition and the continuity
\[
b(\boldsymbol{\beta},\boldsymbol{\omega})\le C_2||\boldsymbol{\beta}||\,||\boldsymbol{\omega}||\quad\forall
(\boldsymbol{\beta},\boldsymbol{\omega})\in H^{-1/2}_\perp\mathsf{\Lambda}^q(\mathrm{d},\Gamma_h)\times H^{-1/2}_\perp\mathsf{\Lambda}^r(\mathrm{d},\Gamma_h)\,,
\]
as a standard result of Babu\v{s}ka-Brezzi theory.
\end{proof}
\section{Numerical experiment}\label{sec_experiment}
\subsection{Description of the test case}\label{sec_test_case}
\begin{figure}
\centering
\includegraphics[width=0.5\textwidth]{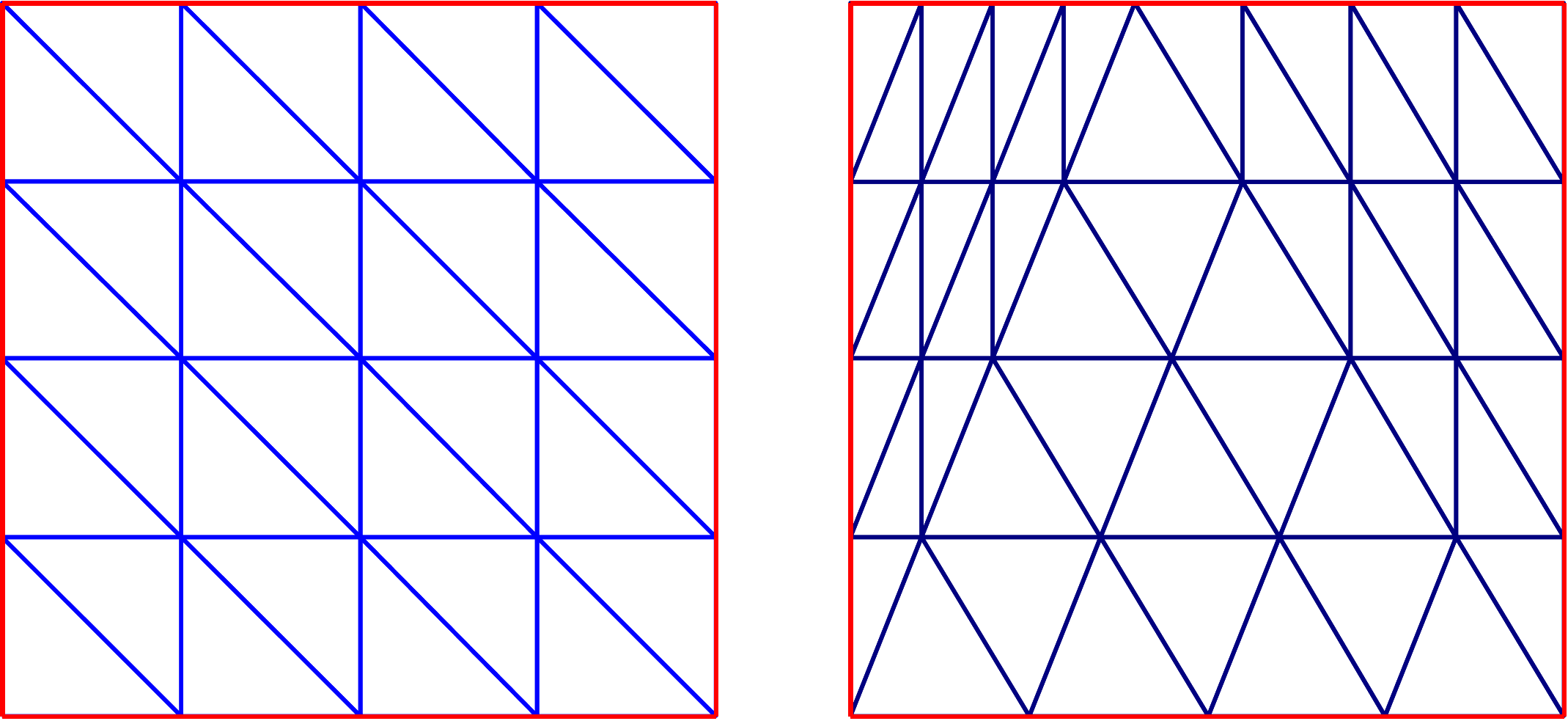}
\caption{\label{fig_Test_Case}Two different initial meshes on a unit square. The left mesh is associated with index $i$ (source mesh), the right one with index $j$ (target mesh).}
\end{figure}
A demonstrator has been implemented in {\sc Matlab}\textsuperscript{\textregistered}, based on the LehrFEM software \cite{Burtscher2013}.\par
The numerical experiment is based on the meshes depicted in Fig.~\ref{fig_Test_Case}. The figure shows the coarsest meshes. Uniform $h$-refinement is accomplished by subdividing each triangle into four triangles.\par
The numerical experiment is concerned with the Whitney complex $(\mathcal{W},\mathrm{d})$, the B-C complex $(\mathcal{B}_0,\mathrm{d})$, and the mesh coupling operators $Q_{ji}^{r,\textnormal{deRham}}$, $Q_{ji}^{r,\textnormal{Galerkin}}$, and $Q_{ji}^{r,\textnormal{B-C}}$. Operator $Q_{ji}^{r,\textnormal{deRham}}$ relies on straightforward interpolation by Whitney and de Rham maps. The other operators were defined in Sect.~\ref{sec_def_operators}. Values $r=0$ and $r=1$ correspond to the scalar and the vector case, respectively.\par
We strive for an implementation of $Q_{ji}^{r,\textnormal{B-C}}$ with linear complexity, by solving \eqref{defopmatrix} for $\{\omega_j\}$. The computation of $[M]_{ji}^r$ involves integration of products of basis functions defined on different meshes. For this purpose, a common triangulation is constructed by a marching front algorithm with linear complexity \cite{Gander2013}. On the common triangulation, standard quadrature rules can be used.\footnote{As an aside, we remark that the same kind of algorithm is available for intersection of tetrahedral meshes in 3D.}\par
\begin{figure}
\centering
\includegraphics[width=0.6\textwidth]{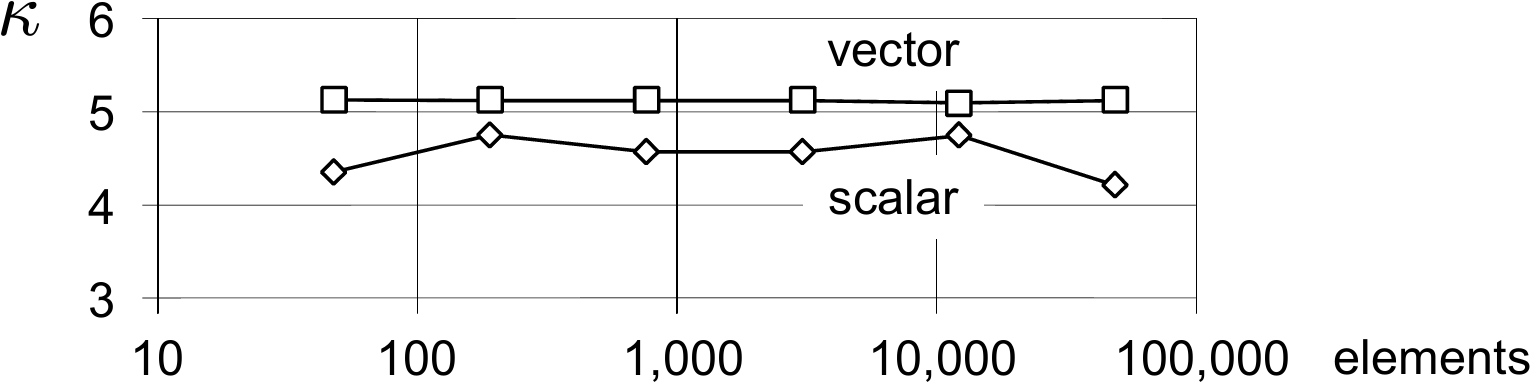}
\caption{\label{fig_Condition_Number}Condition number $\kappa$ of matrix $[M]_{jj}^r$ versus number of elements for regular mesh refinement in the B-C case. For comparison, we report $\kappa\approx 48$ for the Galerkin scalar and $\kappa\approx 20$ for the Galerkin vector case.}
\end{figure}
Eq.~\eqref{defopmatrix} is solved by a few Conjugate-Gradient Squared (CGS) steps \cite{Sonneveld1989}. Fig.~\ref{fig_Condition_Number} shows that the condition number $\kappa$ for $[M]_{jj}^r$ is $\kappa\approx 5$. This resulted in $5\ldots 8$ CGS steps for a relative residual of $10^{-6}$. Since the condition number is independent of regular mesh refinement we maintain linear complexity.\par
We work with smooth scalar and vector data,
\begin{equation}\label{initial}
\boldsymbol{\omega}^0(x,y)=\sin(\pi x)\sin(\pi y)\,,\quad
\boldsymbol{\omega}^1(x,y)=\sin(\pi y)\,\mathrm{d}x+\sin(\pi x)\,\mathrm{d}y\,,
\end{equation}
where $(x,y)$ are Cartesian coordinates on the unit square. The interpolants, nodal and edge, on the source mesh are denoted by $\boldsymbol{\omega}^r_i$.\par
To verify the implementation we numerically compute $\mathrm{d}\,Q_{ji}^1\,\mathrm{d}\,\boldsymbol{\omega}^0_i$, which should be zero in the commuting case, up to round-off. The experiment confirms that $Q_{ji}^{r,\textnormal{deRham}}$ and $Q_{ji}^{r,\textnormal{B-C}}$ enjoy the commuting property, while $Q_{ji}^{r,\textnormal{Galerkin}}$ does not.
\subsection{Results}\label{sec_results}
\subsubsection{Experiment 1: Deterioration of data while being repeatedly mapping back and forth}
\begin{figure}
\centering
\includegraphics[width=0.9\textwidth]{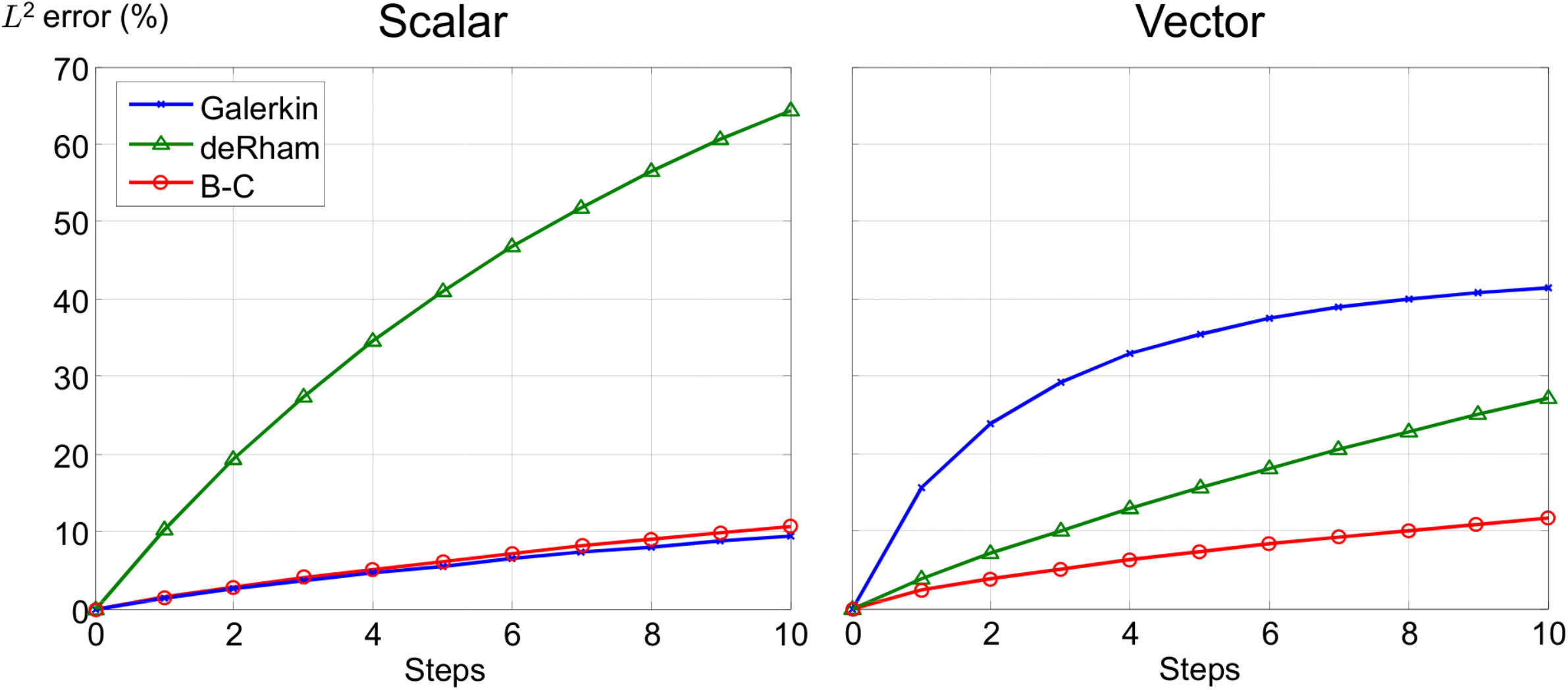}
\caption{\label{fig_Experiment_1}The initial data \eqref{initial} is set on the source mesh. The data deteriorates while being mapped back and forth repeatedly between both meshes. It is well-known that the interpolation method is highly diffusive in the scalar case. Interestingly, this is not observed in the vector case. The B-C-based projection method exhibits the best performance.}
\end{figure}
This experiment is conducted after two successive $h$-refinements of the initial mesh shown in Fig.~\ref{fig_Test_Case}. We pick initial data $\boldsymbol{\omega}^r_i$ according to \eqref{initial}, and map it from source to target mesh. Then, the roles of the meshes are interchanged, and the data is mapped back. These steps are applied repeatedly. The process is diffusive, and we study the deterioration of the data in terms of the relative $L^2$ error after $\nu=0,1,\ldots$ steps,
\[
\mathrm{err}_\nu=\frac{||\bigl((Q_{ij}^rQ_{ji}^r)^\nu-\mathrm{Id}\bigr)\boldsymbol{\omega}_i^r||}{||\boldsymbol{\omega}_i^r||}\cdot 100\%\,,
\]
see Fig.~\ref{fig_Experiment_1}. It is well-known that nodal interpolation is much less accurate than scalar Galerkin projection \cite[p.~292]{Canuto2007}, \cite[Sect.~1.4]{Flemisch2007}. B-C-based projection is similar to Galerkin projection. In the vector case, surprisingly, edge interpolation is more accurate than Galerkin projection. Similar findings have been reported in \cite[Fig.~2, Fig.~4]{Journeaux2014}. The experiment can be recognized as power iteration, an eigenvalue algorithm. The Galerkin approach seems to have a dominant eigenvalue, and the data is quickly reduced to the related eigenvector. In the vector case, B-C-based projection has the best performance, comparable to the scalar case.
\subsubsection{%\texorpdfstring{
Experiment 2: Convergence rates under $h$-refinement}%{Experiment 2: Convergence rates under h-refinement}}
\begin{figure}
\centering
\includegraphics[width=0.9\textwidth]{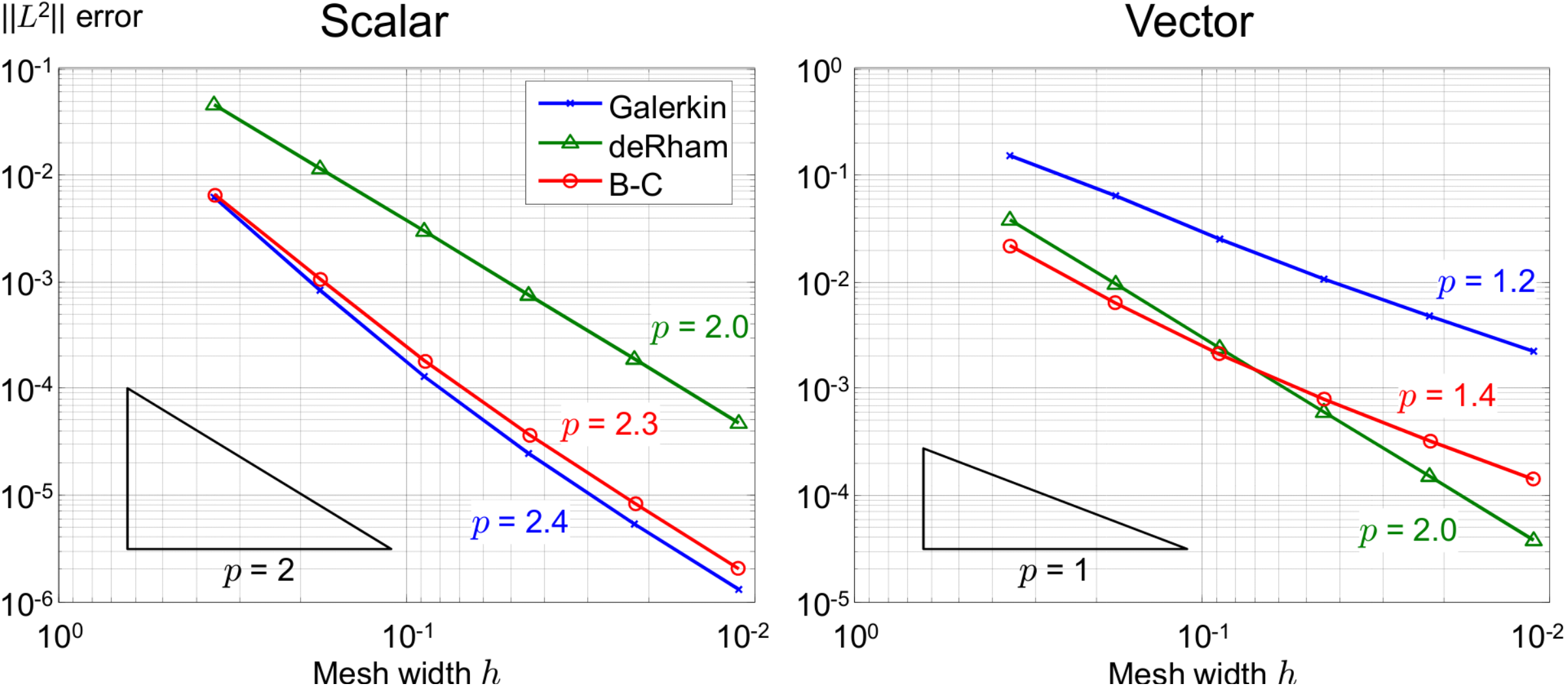}
\caption{\label{fig_Experiment_2_L2}The initial data \eqref{initial} is set on the source mesh, and mapped back and forth once. Convergence of the error in the $L^2$ norm for uniform $h$-refinement is studied. The theoretical convergence rates for smooth data are $p=2$ and $p=1$ in the scalar and vector case, respectively. The indicated convergence rates are based on linear regression.}
\end{figure}
We map back and forth only once, and study the convergence of the result under a uniform $h$-refine\-ment of both meshes, that is,
\[
\mathrm{err}_h=||(Q_{ij}^rQ_{ji}^r-\mathrm{Id})\boldsymbol{\omega}_i^r||\,.
\]
Since the projection error is bounded by the best approximation error, the theoretical convergence rates for smooth data are $p=2$ and $p=1$ in the scalar and vector case, respectively.\footnote{Some caution has to be exerted here, because the Whitney interpolants are not as smooth as required. For a comprehensive discussion about the approximation properties of finite element differential forms see \cite[Sect.~5.4]{Arnold2010}.} The experimental $L^2$ convergence rates are depicted in Fig.~\ref{fig_Experiment_2_L2}, they are as expected, except for the vector de Rham case, where superconvergence seems to occur.\par
\begin{figure}
\centering
\includegraphics[width=0.9\textwidth]{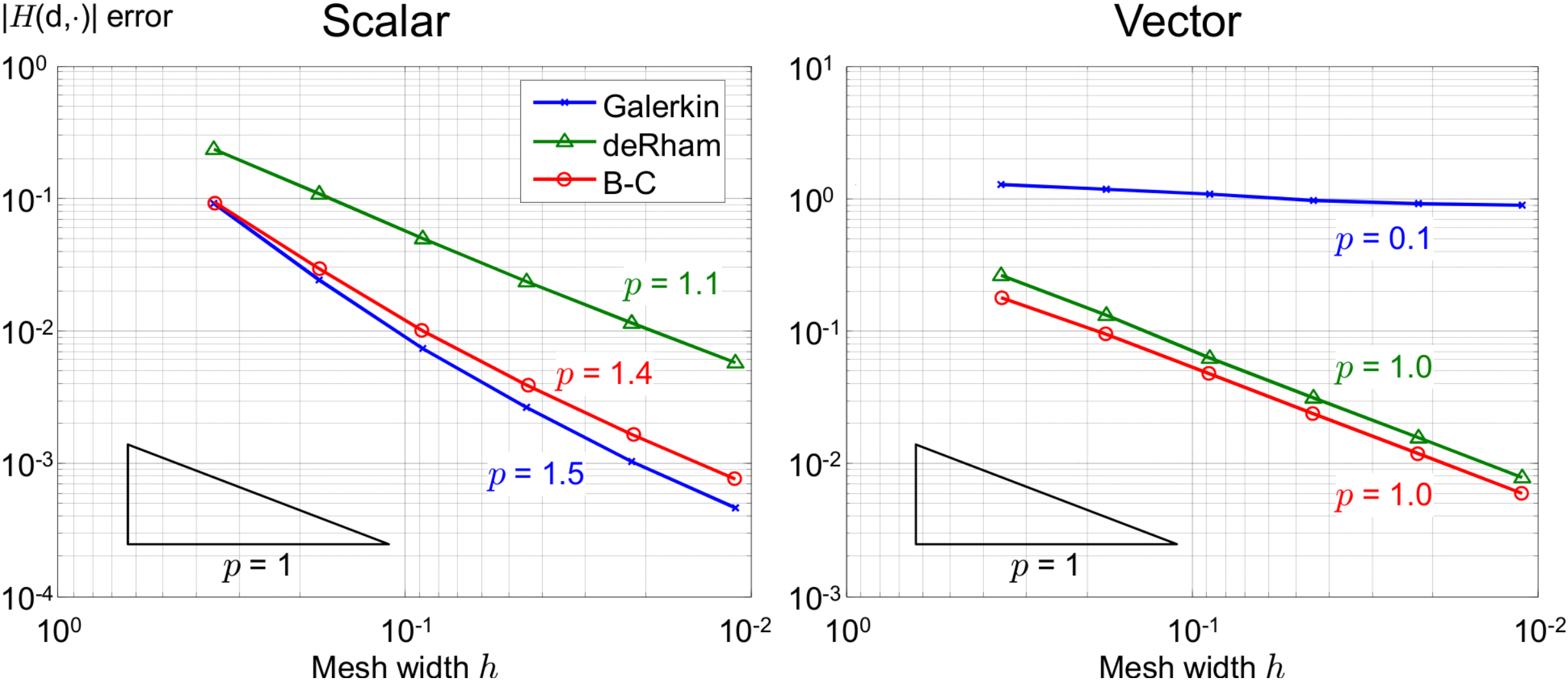}
\caption{\label{fig_Experiment_2_Hd}The initial data \eqref{initial} is set on the source mesh, and mapped back and forth once. Convergence of the error in the $H(\mathrm{d},\cdot)$ seminorm for uniform $h$-refinement is studied. Compared to the $L^2$ norms, one order of convergence is lost by differentiation. Not so in the vector case, where we may benefit from the commuting property. We actually observe the theoretical convergence rate $p=1$ for facet elements in this case.}
\end{figure}
Fig.~\ref{fig_Experiment_2_Hd} shows the $L^2$ convergence rates of the exterior derivatives, that is, the $H(\mathrm{d},\cdot)$ seminorm. For instance, if potential problems are considered, the seminorm measures the convergence of the field results. Nodal elements exhibit second order convergence. Therefore, for their derivative we expect at least first order convergence, and this is confirmed by the experiment. The same argument in the vector case yields a bounded error for the derivative. However, taking benefit from the commuting property we regain one order of convergence. In fact,
\begin{align*}
\mathrm{err}_h&=|(Q_{ij}^1Q_{ji}^1-\mathrm{Id})\boldsymbol{\omega}_i^1|_{H(\mathrm{d},\cdot)}
=||\mathrm{d}(Q_{ij}^1Q_{ji}^1-\mathrm{Id})\boldsymbol{\omega}_i^1||_{L^2}\\
&=||(Q_{ij}^2Q_{ji}^2-\mathrm{Id})\mathrm{d}\boldsymbol{\omega}_i^1||_{L^2}
=||(Q_{ij}^2Q_{ji}^2-\mathrm{Id})\boldsymbol{\omega}_i^2||_{L^2}\,,
\end{align*}
where we let $\boldsymbol{\omega}_i^2=\mathrm{d}\boldsymbol{\omega}_i^1$. This demonstrates that the convergence rate agrees with the $L^2$ convergence rate of facet elements, which is $p=1$.
\section{Conclusions}
After reviewing the state-of-the art of mesh coupling at nonconforming interfaces we introduced the Buffa-Christiansen complex, as well as projection-based mesh coupling operators whose Lagrange multiplier spaces are chosen as B-C spaces. This results in a theorem that states that B-C-based mesh coupling operators are commuting projectors of Whitney spaces. From the theoretical analysis and from the numerical experiment we conclude that the B-C-based approach combines the good properties of simple interpolation and of Galerkin projection: stability (Prop.~\ref{prop_infsup}), structure preservation (Thm.~\ref{thm_struct_preserv}), quasi-optimality (Prop.~\ref{prop_opt}), linear complexity (Sect.~\ref{sec_test_case}), and good accuracy (Sect.~\ref{sec_results}). Moreover, the construction relies on the topology of the mesh only, not on metric information (Remark \ref{rem_nometric}). A recent work paves the way for extension to higher polynomial order finite element differential forms \cite{Smirnova2013}.\par
%\bibliographystyle{amsplain}
%\bibliography{lit1000,papers,local}
\providecommand{\bysame}{\leavevmode\hbox to3em{\hrulefill}\thinspace}
\providecommand{\MR}{\relax\ifhmode\unskip\space\fi MR }
\providecommand{\MRhref}[2]{%
  \href{http://www.ams.org/mathscinet-getitem?mr=#1}{#2}
}
\providecommand{\href}[2]{#2}

\end{document}